\documentclass[a4paper,12pt]{article}
\usepackage{a4wide}
\usepackage{amsmath}
\usepackage{amssymb}
\usepackage{amsthm}
\usepackage{latexsym}
\usepackage{graphicx}
\usepackage[english]{babel}
\usepackage{makeidx}

\newtheorem{obs} [subsection]{Remark}
\newtheorem{exm} [subsection]{Example}

\newtheorem{prop}[subsection]{Proposition}
\newtheorem{conj}[subsection]{Conjecture}
\newtheorem{teor}[subsection]{Theorem}
\newtheorem{lema}[subsection]{Lemma}
\newtheorem{cor} [subsection]{Corollary}
\newcommand{\Zng}{$\mathbb Z^n$-graded $S$-module}

\def\sdepth{\operatorname{sdepth}}
\def\depth{\operatorname{depth}}

\begin{document}
\selectlanguage{english}
\frenchspacing

\large
\begin{center}
\textbf{On the Stanley depth of edge ideals of line and cyclic graphs}

Mircea Cimpoea\c s
\end{center}
\normalsize

\begin{abstract}
We prove that the edge ideals of line and cyclic graphs and their quotient rings satisfy the Stanley conjecture. We compute the Stanley depth for the quotient ring of the edge ideal associated to a cycle graph of length $n$, given a precise formula for $n\equiv 0,2 \pmod{3}$ and tight bounds for $n\equiv 1 \pmod{3}$. Also, we give bounds for the Stanley depth of a quotient of two monomial ideals, in combinatorial terms.


\end{abstract}

\section*{Introduction}

Let $K$ be a field and $S=K[x_1,\ldots,x_n]$ the polynomial ring over $K$.
Let $M$ be a \Zng. A \emph{Stanley decomposition} of $M$ is a direct sum $\mathcal D: M = \bigoplus_{i=1}^rm_i K[Z_i]$ as a $\mathbb Z^n$-graded $K$-vector space, where $m_i\in M$ is homogeneous with respect to $\mathbb Z^n$-grading, $Z_i\subset\{x_1,\ldots,x_n\}$ such that $m_i K[Z_i] = \{um_i:\; u\in K[Z_i] \}\subset M$ is a free $K[Z_i]$-submodule of $M$. We define $\sdepth(\mathcal D)=\min_{i=1,\ldots,r} |Z_i|$ and $\sdepth_S(M)=\max\{\sdepth(\mathcal D)|\;\mathcal D$ is a Stanley decomposition of $M\}$. The number $\sdepth_S(M)$ is called the \emph{Stanley depth} of $M$. In \cite{apel}, J.\ Apel restated a conjecture firstly given by Stanley in \cite{stan}, namely that $\sdepth_S(M)\geq\depth_S(M)$ for any \Zng $\;M$. This conjecture proves to be false, in general, for $M=S/I$ and $M=J/I$, where $I\subset J\subset S$ are monomial ideals, see \cite{duval}.

Herzog, Vladoiu and Zheng show in \cite{hvz} that $\sdepth_S(M)$ can be computed in a finite number of steps if $M=I/J$, where $J\subset I\subset S$ are monomial ideals.  However, it is difficult to compute this invariant, even in some very particular cases. In \cite{rin}, Rinaldo give a computer implementation for this algorithm, in the computer algebra system $CoCoA$ \cite{cocoa}. However, it is difficult to compute this invariant, even in some very particular cases.  For instance in \cite{par} Biro et al. proved that $\sdepth(m)= \left\lceil n/2 \right\rceil$ where $m=(x_1,\ldots,x_n)$.


Let $I_n$ and $J_n$ be the edges ideals associated to the $n$-line, respectively $n$-cycle, graph. Firstly, we prove that $\depth(S/J_n) = \left\lceil \frac{n-1}{3} \right\rceil$, see Proposition $1.3$. Alin \c Stefan \cite{alin} proved that $\sdepth(S/I_n) = \left\lceil \frac{n}{3} \right\rceil$. Using similar techniques, we prove that $\sdepth(S/J_n) =  \left\lceil \frac{n-1}{3} \right\rceil$, for $n\equiv 0 \pmod{3}$ and $n\equiv 2 \pmod{3}$. Also, we prove that $\left\lceil \frac{n-1}{3} \right\rceil \leq \sdepth(S/J_n) \leq \left\lceil \frac{n}{3} \right\rceil$, for $n\equiv 1 \pmod{3}$. See Theorem $1.9$. In particular, $S/J_n$ satisfies the Stanley conjecture. Also, we note that both $I_n$ and $J_n$ satisfy the Stanley conjecture, see Corollary $1.5$. In Proposition $1.10$, we prove that $\sdepth(J_n/I_n) = \depth(J_n/I_n) = \left\lceil \frac{n+2}{3} \right\rceil$. In the second section, we give an upper bound for the Stanley depth of a quotient of two square free monomial ideals, in combinatorial terms, see Theorem $2.4$. Also, we give a lower bound for the Stanley depth of a quotient of two arbitrary monomial ideals, see Proposition $2.9$.


\footnotetext[1]{We greatfully acknowledge the use of the computer algebra system CoCoA (\cite{cocoa}) for our experiments.}
\footnotetext[2]{The support from grant ID-PCE-2011-1023 of Romanian Ministry of Education, Research and Innovation is gratefully acknowledged.}
\newpage
\section{Main results}

Let $n\geq 3$ be an integer and let $G=(V,E)$ be a graph with the vertex set $V=[n]$ and edge set $E$. Then the \emph{edge ideal $I(G)$} associated to $G$ is the squarefree monomial ideal $I=(x_ix_j:\;\{i,j\}\in E)$ of $S$.

We consider the l\emph{ine graph $L_n$} on the vertex set $[n]$ and with the edge set $E(L_n)=\{(i,i+1):\;i\in[n-1]\}$. Then $I_n=I(L_n)=(x_1x_2,\ldots,x_{n-1}x_n)\subset S$. Also, we consider the cyclic graph $C_n$ on the vertex set $[n]$ and with the edge set $E(C_n)=\{(i,i+1):\;i\in[n-1]\}\cup\{(n,1)\}$. Then $J_n=I_n+(x_nx_1) \subset S$.

We recall the well known Depth Lemma, see for instance \cite[Lemma 1.3.9]{real} or \cite[Lemma 3.1.4]{vasc}.

\begin{lema}(Depth Lemma)
If $0 \rightarrow U \rightarrow M \rightarrow N \rightarrow 0$ is a short exact sequence of modules over a local ring $S$, or a Noetherian graded ring with $S_0$ local, then

a) $\depth M \geq \min\{\depth N,\depth U\}$.

b) $\depth U \geq \min\{\depth M,\depth N +1 \}$.

c) $\depth N\geq \min\{\depth U - 1,\depth M\}$.
\end{lema}

Using Depth Lemma, Morey proved in \cite{mor} the following result.

\begin{lema}\cite[Lemma 2.8]{mor}
$\depth(S/I_n)=\left\lceil \frac{n}{3} \right\rceil$.
\end{lema}

In the following, we will prove a similar result for $S/J_n$.

\begin{prop}
$\depth(S/J_n)=\left\lceil \frac{n-1}{3} \right\rceil$.
\end{prop}

\begin{proof}
We denote $S_k:=K[x_1,\ldots,x_k]$, the ring of polynomials in $k$ variables. We use induction on $n$. If $n\leq 3$ then is an easy exercise to prove the formula. Assume $n\geq 4$ and consider the short exact sequence
\[ 0 \longrightarrow S/(J_n:x_n) \stackrel{\cdot x_n}{\longrightarrow} S/J_n \longrightarrow S/(J_n,x_n) \longrightarrow 0. \]
Note that $(J_n:x_n) = (x_1,x_{n-1}, x_2x_3,\ldots,x_{n-3}x_{n-2})$ and therefore we get $S/(J_n:x_n)\cong \linebreak K[x_2,\ldots,x_{n-2},x_n]/(x_2x_3,\ldots,x_{n-3}x_{n-2})\cong (S_{n-3}/I_{n-3})[x_n]$. 

Also, $(J_n, x_n) = (x_1x_2,\ldots,x_{n-2}x_{n-1},x_n)$ and therefore $S/(J_n,x_n) \cong S_{n-1}/I_{n-1}$. By Lemma $1.2$, we get
$\depth(S/(J_n:x_n)) = \left\lceil \frac{n-3}{3} \right\rceil + 1 = \left\lceil \frac{n}{3} \right\rceil$ and $\depth(S/(J_n,x_n)) =\left\lceil \frac{n-1}{3} \right\rceil$. If $n\equiv 0 \pmod{3}$ or $n\equiv 2 \pmod{3}$, then 
$\left\lceil \frac{n-1}{3} \right\rceil = \left\lceil \frac{n}{3} \right\rceil$, and, by using Lemma $1.1$, we get $\depth(S/J_n)=\left\lceil \frac{n-1}{3} \right\rceil$, as required.

Assume $n\equiv 1 \pmod{3}$. We claim that we have the $S$-module isomorphism
\[\frac{(J_n:x_n)}{J_n} \cong x_{n-1}\left(\frac{K[x_1,\ldots,x_{n-3}]}{(x_1x_2,\ldots,x_{n-4}x_{n-3})} \right)[x_{n-1}] \oplus x_{1} \left( \frac{K[x_3,\ldots,x_{n-2}]}{(x_3x_4,\ldots,x_{n-3}x_{n-2})} \right) [x_{1}]. \]
Indeed, if $u\in (J_n:x_n)$ is a monomial such that $u\notin J_n$, then $x_1|u$ or $x_{n-1}|u$. If $x_{n-1}|u$, then $u=x_{n-1}v$ with $v\in S$. 

Since $u\notin J_n$, it follows that $v=x_{n-1}^{\alpha}w$, with $\alpha\geq 1$, $w\in K[x_1,\ldots,x_{n-3}]$ and $w\notin (x_1x_2,\ldots,x_{n-4}x_{n-3})$. Similarly, if $x_{n-1}\nmid u$, then $x_{1}|u$ and $u=x_1^{\alpha}w$ with $\alpha\geq 1$, $w\in K[x_3,\ldots,x_{n-2}]$ and $w\notin (x_3x_4,\ldots,x_{n-3}x_{n-2})$.

Using the above isomorphism and Lemma $1.2$, it follows that $$\depth \left( \frac{(J_n:x_n)}{J_n} \right)= \depth\left( \frac{K[x_3,\ldots,x_{n-2}]}{(x_3x_4,\ldots,x_{n-3}x_{n-2})} \right) + 1 = \left\lceil \frac{n-4}{3} \right\rceil + 1 = \left\lceil \frac{n-1}{3} \right\rceil.$$ Now, using Lemma $1.1$ for the short exact sequence $0 \rightarrow \frac{(J_n:x_n)}{J_n} \rightarrow S/J_n \rightarrow S/(J_n:x_n) \rightarrow 0$, we are done.
\end{proof}

Note that the previous Proposition can be seen as a consequence of \cite[Proposition 5.0.6]{bou}. However, we preferred to give a direct proof in order to relate it with the Stanley depth case. Now, we recall the following result of Okazaki.

\begin{teor}\cite[Theorem 2.1]{okazaki}
Let $I\subset S$ be a monomial ideal (minimally) generated by $m$ monomials. Then: \[ \sdepth(I) \geq \max\{1,n-\left\lfloor \frac{m}{2} \right\rfloor\}.\]
\end{teor}

As a direct consequence of Lemma $1.2$, Proposition $1.3$ and Theorem $1.4$, we get.

\begin{cor}
$\sdepth(I_n)\geq 1+\frac{n-1}{2}$ and $\sdepth(J_n)\geq \frac{n}{2}$. In particular, $I_n$ and $J_n$ satisfy
the Stanley conjecture.
\end{cor}

In \cite{alin}, Alin \c Stefan computed the Stanley depth for $S/I_n$.

\begin{lema}\cite[Lemma 4]{alin}
$\sdepth(S/I_n) = \left\lceil \frac{n}{3} \right\rceil$.
\end{lema}

In \cite{asia}, Asia Rauf proved the analog of Lemma $1.1(a)$ for $\sdepth$:

\begin{lema}
Let $0 \rightarrow U \rightarrow M \rightarrow N \rightarrow 0$ be a short exact sequence of $\mathbb Z^n$-graded $S$-modules. Then:
\[ \sdepth(M) \geq \min\{\sdepth(U),\sdepth(N) \}. \]
\end{lema}

Using these lemmas, we are able to prove the following Proposition.

\begin{prop}
$\sdepth(S/J_n) \geq \left\lceil \frac{n-1}{3} \right\rceil$. In particular, $S/J_n$ satisfies the Stanley conjecture.
\end{prop}

\begin{proof}
As in the proof of Proposition $1.3$, we consider the short exact sequence
\[ 0 \longrightarrow S/(J_n:x_n) \stackrel{\cdot x_n}{\longrightarrow} S/J_n \longrightarrow S/(J_n,x_n) \longrightarrow 0. \]
Since $S/(J_n:x_n)\cong (S_{n-2}/I_{n-2})[x_n]$ and $S/(J_n,x_n)\cong S_{n-1}/I_{n-1}$, by Lemma $1.6$ and \cite[Lemma 3.6]{hvz}, we get 
$\sdepth(S/(J_n:x_n)) = \left\lceil \frac{n-3}{3} \right\rceil + 1 = \left\lceil \frac{n}{3} \right\rceil$ and $\sdepth(S/(J_n,x_n)) = \left\lceil \frac{n-1}{3} \right\rceil$. Using Lemma $1.7$, we get $\sdepth(S/J_n)\geq \left\lceil \frac{n-1}{3} \right\rceil$, as required.
\end{proof}

Let $\mathcal P\subset 2^{[n]}$ be a poset and $\mathbf P:\mathcal P=\bigcup_{i=1}^r [F_i,G_i]$ be a partition of $\mathbf P$. We denote $\sdepth(\mathbf P):=\min_{i\in [r]} |D_i|$. Also, we define the Stanley depth of $\mathcal P$, to be the number
\[\sdepth(\mathcal P) = \max\{\sdepth(\mathbf P):\; \mathbf P \; \emph{is\; a\; partition\; of} \; \mathcal P \}.\]

We recall the method of Herzog, Vladoiu and Zheng \cite{hvz} for computing the Stanley depth of $S/I$ and $I$, where $I$ is a squarefree monomial ideal. Let $G(I)=\{u_1,\ldots,u_s\}$ be the set of minimal monomial generators of $I$. We define the following two posets:
\[ \mathcal P_I:=\{\sigma \subset [n]:\; u_i|x_{\sigma}:=\prod_{j\in\sigma}x_j \;\emph{for\;some}\;i\;\}\;\emph{and}\; 
\mathcal P_{S/I}:=2^{[n]}\setminus \mathcal P_I. \]
Herzog Vladoiu and Zheng proved in \cite{hvz} that $\sdepth(I)=\sdepth(\mathcal P_{I})$ and $\sdepth(S/I)=\sdepth(\mathcal P_{S/I})$.
Now, for $d\in\mathbb N$ and $\sigma\in \mathcal P$, we denote
\[ \mathcal P_d = \{\tau\in\mathcal P\;:\; |\tau|=d \}\;,\;\mathcal P_{d,\sigma} = \{ \tau\in\mathcal P_d\;:\; \sigma\subset\tau \}. \]

With these notations, we are able to prove the following result.

\begin{teor}
(1) $\sdepth(S/J_n)= \left\lceil \frac{n-1}{3} \right\rceil$, for $n\equiv 0 \pmod{3}$ and $n\equiv 2 \pmod{3}$.

(2) $\sdepth(S/J_n) \leq \left\lceil \frac{n}{3} \right\rceil$, for $n\equiv 1 \pmod{3}$.
\end{teor}

\begin{proof}
Using Proposition $1.8$, it is enough to prove the "$\leq$" inequalities. Let $\mathcal P=\mathcal P_{S/J_n}$. Firstly, note that if $\sigma \in \mathcal P$ such that $P_{d,\sigma}=\emptyset$, then $\sdepth(\mathcal P)<d$. Indeed, let $\mathbf P:\mathcal P=\bigcup_{i=1}^r [F_i,G_i]$ be a partition of $\mathcal P$ with $\sdepth(\mathcal P)=\sdepth(\mathbf P)$. Since $\sigma\in\mathcal P$, it follows that  $\sigma \in [F_i,G_i]$ for some $i$. If $|G_i|\geq d$, then it follows that $\mathcal P_{\sigma,d}\neq\emptyset$, since there are subsets in the interval $[F_i,G_i]$ of cardinality $d$ which contain $\sigma$, a contradiction. Thus, $|G_i|<d$ and therefore $\sdepth(\mathcal P)<d$.

We have three cases to study.

1. If $n=3k\geq 3$ and $\sigma=\{1,4,\ldots,3k-2\}$, then $\mathcal P_{k+1,\sigma} = \emptyset$. Indeed, if $u=x_1x_4\cdots x_{3k-2}$, one can easily see that $u\cdot x_j\in J_n$ for all $j\in [n]\setminus\sigma$. Therefore, be previous remark, $\sdepth(S/J_n)=\sdepth(\mathcal P)\leq k = \left\lceil \frac{n-1}{3} \right\rceil$, as required.

2. If $n=3k+2\geq 5$ and $\sigma=\{1,4,\ldots,3k+1\}$, then $\mathcal P_{k+2,\sigma} = \emptyset$. As above, it follows that $\sdepth(S/J_n)\leq k+1 = \left\lceil \frac{n-1}{3} \right\rceil$.

3. If $n=3k+1\geq 7$ and $\sigma=\{1,4,\ldots,3k-2,3k\}$, then 
$\mathcal P_{k+2,\sigma} = \emptyset$ and therefore $\sdepth(\mathcal P)\leq k+1 = \left\lceil \frac{n}{3} \right\rceil$.
\end{proof}

\begin{prop}
$\sdepth(J_n/I_n) = \depth(J_n/I_n) = \left\lceil \frac{n+2}{3} \right\rceil$, for all $n\geq 3$.
\end{prop}

\begin{proof}
One can easily check that $\frac{J_3}{I_3} \cong x_1x_3K[x_1,x_3]$. Thus $\sdepth(J_3/I_3) = \depth(J_3/I_3) = 2$, as required. Similarly, for $n=4$, we have $\frac{J_4}{I_4} \cong x_1x_4K[x_1,x_4]$ and for $n=5$, we have $\frac{J_5}{I_5} \cong x_1x_5K[x_1,x_3,x_5]$.

Now, assume $n\geq 6$, and let $u\in J_n$ a monomial such that $u\notin I_n$. It follows that $u=x_1x_nv$, with $v\in K[x_1,x_3,\ldots,x_{n-2},x_n]$. We can write $v=x_1^{\alpha}x_n^{\beta}w$, with $w\in K[x_3,\ldots,x_{n-2}]$. Since $u\notin I_n$, it follows that $w\notin (x_3x_4,\ldots,x_{n-3}x_{n-2})$. Therefore, we have the $S$-module isomorphism:
\[ \frac{J_n}{I_n} = x_1x_n \left( \frac{K[x_3,\ldots,x_{n-2}]}{(x_3x_4,\ldots,x_{n-3}x_{n-2})} \right) [x_1,x_n] \]
and therefore, by Lemma $1.2$, Lemma $1.6$ and \cite[Lemma 3.6]{hvz}, we get 
$\sdepth(J_n/I_n) = \depth(J_n/I_n) = \left\lceil \frac{n-4}{3} \right\rceil + 2 = \left\lceil \frac{n+2}{3} \right\rceil$.
\end{proof}

\begin{obs}
\emph{If $n=4$, one can easily see that $\sdepth(S/J_4)=1$. Also, for $n=7$, we can show that $\sdepth(S/J_7)=2$, see Example $2.5$. On the other hand, using the $SdepthLib.coc$ of $CoCoA$, see \cite{rin}, we get $\sdepth(S/J_{10})=4$ and $\sdepth(S/J_{13})=5$. This remark, yields the following conjecture.}
\end{obs}

\begin{conj}
$\sdepth(S/J_n)=\left\lceil \frac{n}{3} \right\rceil$, for all $n\geq 10$ with $n\equiv 1 \pmod{3}$.
\end{conj}

Even if $J_n$ and $I_n$ are closely related, the difficulty of Conjecture $1.12$ should not be underestimate. See for instance \cite{par}, where the authors, using fine tools of combinatorics were hardly able to compute the Stanley depth of the maximal monomial ideal $(x_1,\ldots,x_n)$. In the second section we will give a possible approach to this problem, see Example $2.5$.

\section{Bounds for Sdepth of quotient of monomial ideals}

\begin{lema}
Let $n\geq 1$ and $0\leq k\leq n$ be two integers and let $\mathcal P=\{\sigma\in 2^{[n]}\;|\;|\sigma|\leq k \}$. Then, there exists a partition $\mathbf P: \mathcal P = \bigcup_{i=1}^r [C_i,D_i]$ with $|D_i|=k$.
\end{lema}

\begin{proof}
If $k=n$ or $k=0$ there is nothing to prove. Assume $1\leq k \leq n-1$. Note that $\mathcal P$ is the partition associated to $S/I_{n,k+1}$, where $I_{n,k+1}$ is the ideal generated by all the square free monomials of degree $k+1$.
According to \cite[Theorem 1.1]{mirver}, $\sdepth(S/I_{n,k+1})=k$. Thus, we can find a partition of $\mathcal P$, as required.
\end{proof}

\begin{prop}
Let $\mathcal P \subset 2^{[n]}$ be a poset such that $\sdepth(\mathcal P)\geq k$. Then there exists a partition
of $\mathcal P$, such that, for each interval $[C,D]$ of it, if $|C|<k$ then $|D|=k$. 

In particular, the above assertion holds, if $I\subset J$ are two monomial square-free ideals such that $\sdepth(J/I)=k$  and $\mathcal P = \mathcal P_{J/I}:=\mathcal P_{S/I}\cap \mathcal P_{J}$.
\end{prop}

\begin{proof}
According to Herzog, Vladoiu and Zheng \cite{hvz}, we have $\sdepth(J/I)=\sdepth(\mathcal P_{J/I})$. Since $\sdepth(\mathcal P)\geq k$, we can find a partition of $\mathcal P$, such that each interval $[C,D]$ in this partition has $|D|\geq k$. 

Let $[C,D]$ be an interval of the partition of $\mathcal P$. If $|C|\geq s$ or $|D|=s$ there is nothing to do. Assume $|C|<k$ and $|D|>k$. We denote $|C|=t$ and $|D|=s$. Without losing the generality, we may assume that $D=[s]$ and $C=[s]\setminus [s-t]$. Using the previous Lemma, we can find a partition of $[\emptyset,[s-t]]=\bigcup_{i=1}^r [\overline C_i,\overline D_i]$ with $|\overline D_i|=k-t$ whenever $|\overline C_i|<k-t$. Let $C_i=C\cup \overline C_i$ and $D_i=C\cup \overline D_i$. It follows that $[C,D]=\bigcup_{i=1}^r [C_i,D_i]$ is a partition with $|D_i|=k$, whenever $|C_i|<k$. If we apply this method for each interval in the partition of $\mathcal P$, finally, we will get a partition of $\mathcal P$, as required.
\end{proof}

\begin{cor}
Let $\mathcal P \subset 2^{[n]}$ be a poset such that $\sdepth(\mathcal P) \geq k$. Denote $\mathcal P_{\leq k} = \{\sigma\in\mathcal P\: |\sigma|\leq k\}$. Then $\sdepth(\mathcal P_{\leq k})=k$.
\end{cor}

\begin{proof}
Obviously, $\sdepth(\mathcal P_{\leq k})\leq k$. According to Proposition $2.2$, we can find a partition $\mathbf P: \mathcal P=\bigcup_{i=1}^r [F_i,G_i]$ of $\mathcal P$ such that $|G_i|=k$, whenever $|F_i|<k$. Note that 
$$[F_i,G_i]\cap \mathcal P_{\leq k} = \begin{cases}
  [F_i,G_i],\;|F_i|<k, \\
  [F_i,F_i],\;|F_i|=k, \\
  \emptyset,\;|F_i|>k
 \end{cases}$$
Therefore, $\mathcal P_{\leq k}=\bigcup_{i=1}^r [F_i,G_i]\cap \mathcal P_{\leq k}$ is a partition of $\mathcal P_{\leq k}$ with its Stanley depth $\geq k$.
\end{proof}

Let $\mathcal P \subset 2^{[n]}$ be a poset such that $\sdepth(\mathcal P) \geq k$. We denote $\beta_t=|\{\sigma\in\mathcal P:\;|\sigma|=t \}|$, for all $0\leq t\leq k$. 

We consider the poset $\mathcal P_{\leq k}:=\{\sigma\in\mathcal P\;:\;|\sigma|\leq k\}$. By Corollary $2.3$, we can find a partition $\mathbf P:\;\mathcal P_{\leq k}=\bigcup_{i=1}^r [F_i,G_i]$ with $|G_i|=k$ for all $i$. We may assume that $|F_{i}|\leq |F_{i+1}|$ for all $i\leq r-1$. For all $0\leq j\leq k$, we denote $\alpha_j=|\{i\;:\;|F_i|=j\}|$. Let $[F,G]$ be an arbitrary interval in the partition $\mathbf P$ such that $|F|=j$ for some $j\leq k$. Note that in the interval $[F,G]$ we have exactly $\binom{k-j}{t-j}$ sets of cardinality $t$. Therefore, we get $\beta_t = \sum_{j=0}^t \binom{k-j}{t-j} \alpha_j$, for all $0\leq t\leq k$. Moreover, $\alpha_0=\beta_0$, $\alpha_1=\beta_1-k\beta_0$, $\alpha_2=\beta_2-\binom{k}{2}\alpha_0-(k-1)\alpha_1$ and so on. Thus, we proved the following Theorem.

\begin{teor}
If $\sdepth(\mathcal P) \geq k$, then $\alpha_t\geq 0$ for all $0\leq t\leq k$, where $\alpha_0=\beta_0$ and $\alpha_t = \beta_t - \sum_{j=0}^{t-1} \binom{k-j}{t-j} \alpha_j$.
\end{teor}

Note that the above theorem give an upper bound for $\sdepth(J/I)$, where $I\subset J$ are square free monomial ideals. Indeed, we can consider the poset $\mathcal P:=\mathcal P_{J/I}$.

\begin{exm}
\emph{We consider the poset $\mathcal P:=\mathcal P_{S/J_n}$, where $J_n=(x_1x_2,\ldots,x_{n-1}x_n,x_nx_1)\subset S$. We claim that $\beta_t = \binom{n-t+1}{t} - \binom{n-t-1}{t-2}$, for all $0\leq t\leq n$. }

\emph{Indeed, if $\sigma=\{i_1,\ldots,i_t\}\in \mathcal P$ is a set of cardinality $t$ such that $1\leq i_1<i_2<\cdots<i_t\leq n$, then $i_{j+1}\geq i_j+2$ and $\{i_1,i_k\}\neq \{1,n\}$. There are exactly $\binom{n-t+1}{t}$, $t$-tuples $1\leq i_1<i_2<\ldots<i_t\leq n$ with $i_{j+1}\geq i_j+2$ and exactly $\binom{n-t-1}{t-2}$, $t$-tuples $1 = i_1<i_2 < \cdots <i_t=n$ with $i_{j+1}\geq i_j+2$. (To be more clear, if we denote $l_j := i_j - j + 1$, we have $1\leq l_1\leq l_2 \leq \cdots \leq l_t \leq n-t+1$ with $l_{j+1}>l_j$, and there are exactly $\binom{n-t+1}{t}$, $t$-tuples like this. If we fix $l_1=1$ and $l_t=n-t+1$, we have $2\leq l_2\leq\cdots\leq l_{t-1}\leq n-t$ and there are exactly $\binom{n-t-1}{t-2}$, $t-2$-tuples like this).}

\emph{Now, for $n=7$, one can easily check that $\beta_0=1$, $\beta_1=7$, $\beta_2=14$ and $\beta_3=7$. For $k=3$, we have
$\alpha_0=1$, $\alpha_1=4$, $\alpha_2=2$ and $\alpha_3=-1$. This shows, in the light of Theorem $2.4$, that we cannot find a decomposition of the poset associated to $S/J_7$ with its Stanley depth equal to $3$. On the other hand, by Proposition $1.8$, we have $\sdepth(S/J_7)\geq 2$, and thus $\sdepth(S/J_{7})=2$.}

\emph{For $n=3k-2$, where $k\geq 4$, we expect that $\alpha_0,\ldots,\alpha_k$ are nonnegative, which is indeed the case for small values of $k$, using computer experimentation. However, this is useful only as an heuristic method to estimate the Stanley depth of $S/J_n$. In order to compute exactly this invariant, one has to produce a concrete partition of the associated poset.}
\end{exm}




In the second part of this section, we give a lower bound for the Stanley depth of a quotient of monomial ideals in terms of the minimal number of monomial generators. First, we recall several results.

\begin{prop}\cite[Proposition 1.2]{mir}
Let $I\subset S$ be a monomial ideal (minimally) generated by $m$ monomials. Then $\sdepth(S/I)\geq n-m$.
\end{prop}

\begin{prop}\cite[Remark 2.3]{mirci}
Let $I,J\subset S$ be two monomial ideals. Then \linebreak $\sdepth((I+J)/I) \geq \sdepth(J) + \sdepth(S/I) - n.$
\end{prop}

\begin{lema}
Let $I,L\subset S$ be two monomial ideals such that $L$ is minimally generated by some monomials $w_1,\ldots,w_s$ which are not in $I$. Then $\mathcal B = \{w_1+I,\ldots,w_s+I\}$ is a system of generators of $J/I$, where $J:=L+I$.
\end{lema}

\begin{proof}
Denoting $G(I)=\{v_1,\ldots,v_p\}$, it follows that $J = (v_1,\ldots,v_p,w_1,\ldots,w_r)$.
So, if $w\in J\setminus I$ is a monomial, then $w_j|w$ for some $j\in [r]$ and therefore $\mathcal B$ is a system of generators for $J/I$. On the other hand, since $w_1,\ldots,w_r$ minimally generated $L$, we get the minimality of $\mathcal B$.
\end{proof}

We consider $I\subset J \subset S$ two monomial ideals. Denote $G(I)=\{v_1,\ldots,v_p\}$ and $G(J)=\{u_1,\ldots,u_q\}$ the sets of minimal monomial generators of $I$ and $J$.

 If $u_1\in I$, then we may assume that $v_1|u_1$. On the other hand, $I\subset J$ and therefore, there exists an index $i$ such that $u_i|v_1$. We get $u_i|u_1$ and thus $u_i=u_1=v_1$. Using the same argument, we can assume that there exists an integer $r\geq 0$ such that $u_1=v_1,\ldots,u_r=v_r$ and $u_{r+1},\ldots,u_q\notin I$. By Lemma $2.8$, $\{u_{r+1}+I,\ldots,u_q+I\}$ is a set of generators of $J/I$. With these notations, we have the following result, which is similar to \cite[Theorem 2.4]{mircea}.

\begin{prop}
$\sdepth(J/I) \geq n - p - \left\lfloor  \frac{q-r}{2} \right\rfloor$.
\end{prop}

\begin{proof}
Denote $J'=(u_{r+1},\ldots,u_q)$. By our assumptions, we have $J/I = (I+J')/I$. By Proposition $2.7$, it follows that
$\sdepth(J/I)\geq \sdepth(J')+\sdepth(S/I) - n$. By Theorem $1.4$ and Proposition $2.6$ we are done.
\end{proof}


\vspace{2mm} \noindent {\footnotesize
\begin{minipage}[b]{15cm}
Mircea Cimpoea\c s, Simion Stoilow Institute of Mathematics, Research unit 5, P.O.Box 1-764,\\
Bucharest 014700, Romania\\
E-mail: mircea.cimpoeas@imar.ro
\end{minipage}}

\end{document}